\documentclass[12pt]{amsart}
\usepackage{latexsym}
\usepackage{amsfonts}
\usepackage{graphicx}
\usepackage{amsfonts}
\usepackage{amsmath,amscd}
\usepackage{color}
\usepackage[colorlinks=true, pdfstartview=FitV, linkcolor=blue, citecolor=blue, urlcolor=blue]{hyperref}
\usepackage{epsfig}
\usepackage{amsthm,amsfonts}
\usepackage{amssymb,graphicx,color}
\usepackage[all]{xy} 
\usepackage{verbatim}
\usepackage{hyperref}

\def\cal{\mathcal}

\newcommand{\field}[1]{\mathbb{#1}}
\newcommand{\C}{\field{C}}

\newcommand{\N}{\field{N}}

\newcommand{\R}{\field{R}}

\newcommand{\Sing}{\mathrm{Sing }}

\newtheorem{definition}{Definition}[section]

\newtheorem{lemma}[definition]{Lemma}
\newtheorem{theorem}[definition]{Theorem}
\newtheorem{corollary}[definition]{Corollary}
\newtheorem{proposition}[definition]{Proposition}

\font\tenmsy=msbm10

\def\Bbb#1{\hbox{\tenmsy#1}} 
\title [Symmetry defect of $n$-dimensional complete intersections]{Symmetry defect of $n$- dimensional complete intersections in $\mathbb C^{2n-1}$} \makeatletter

\@addtoreset{equation}{section}
\makeatother
\thanks{}
\date {\today}

\author{L.R.G. Dias  \& Z. Jelonek} 

\address[L.R.G. Dias] {Faculdade de Matem\'atica, Universidade Federal de Uberl\^andia, Av. Jo\~ao Naves de \'Avila 2121, 1F-153 - CEP: 38408-100, Uberl\^andia, Brasil.}
\email{lrgdias@ufu.br}

\address[Z. Jelonek]{Instytut Matematyczny\\
	Polska Akademia Nauk\\
	\'Sniadeckich 8, 00-956 Warszawa, Poland.}
\email{najelone@cyf-kr.edu.pl}


\subjclass{MSC Classification: 58 K 30, 14 D 06}
 
\begin{document}

\maketitle

\begin{abstract} Let $X, Y \subset \C^{2n-1}$ be  $n$-dimensional strong complete intersections in a general position. In this note, we consider the set of midpoints of chords connecting a point $x \in X$ to a point $y \in Y$. This set is defined as the image of the map 
$\Phi(x,y)=\frac{x+y}{2}.$  Under geometric conditions on  $X$ and $Y$, we prove that the symmetry defect of $X$ and $Y$, which is the bifurcation set $B(X,Y)$ of the mapping $\Phi$, is an algebraic variety, characterized by a topological invariant. We introduce a hypersurface that  approximates the set $B(X,Y)$ and we present an estimate for its degree. Moreover, for any two $n$-dimensional strong complete intersections $X,Y\subset \C^{2n-1}$ (including the case $X=Y$) we introduce a generic symmetry defect set $\tilde{B}(X,Y)$ of $X$ and $Y$, which is defined up to homeomorphism.
\end{abstract}

\section{Introduction}

Over the last two decades numerous methods have been developed to
study affine geometry of surfaces and curves especially their
affinely invariant symmetry characteristics. The symmetry sets
\cite{Bruce, Damon, Damon1} and the center symmetry sets were
investigated extensively in \cite{GH, GJ, Jan}. Several constructions of
the set equivalent to the point of central symmetry for perturbed
centrally symmetric ovals were presented in the literature and
resulted in the kind of symmetry defect called center symmetry set.
The center symmetry set directly appears in the construction of the
so-called Wigner caustic. This caustic is obtained by the stationary
phase method applied to the semiclassical Wigner function which
completely describes a quantum state in the symplectic phase space
\cite{berry}. It is {built} of points where the central symmetry, i.e.
the number of intervals ending in the surface and passing centrally
through that point, changes. We call this set a symmetry defect or
bifurcation set. In \cite{JJR}   this construction was generalized for algebraic
varieties $Z^n\subset \C ^{2n}$.

Let us recall that a pair of points $x, y$ of a plane curve $Z$ is called a {\it parallel pair} if their tangent lines  are parallel. Given $\lambda \in [0, 1]$, the affine $\lambda$-equidistant is the set of points $\lambda x +(1-\lambda)y$ such that $x, y$ is a parallel pair of $Z$. In the special case $\lambda=1/2$, the affine $1/2$-equidistant is   mentioned above the Wigner Caustic.   
 
Related to the affine $1/2$-equidistant, it was proved in \cite{JJR} the following general result: let $Z\subset \C^{2n}$ be a  $n$-dimensional algebraic variety  in a general position. For $a\in \C^{2n}$, let $\mu(a)$ be the number of chords of $Z$ whose midpoint is $a$. Then there exists an algebraic set $B(Z)\subset \C^{2n}$, called {\it symmetric defect} of $Z$, such that $\deg B(Z) <(\deg Z)^2(1+2n(\deg Z-1)) $ and  $\mu(a)$ is constant on $\C^{2n}\setminus  B(Z)$. 

This note  is motivated by the recent results of  \cite{DFJ} and \cite{JJR}. We address the problem of how to introduce a {\it symmetric defect set} associated to $n$-dimensional strong complete intersections $X,Y\subset \C^{2n-1}$. We also investigate algebraic and topological properties of this set.

In order to state our theorems properly, we introduce the following definitions: 

\begin{definition} We say that an $n$-dimensional affine variety $Z\subset \C^m$ is a strong complete intersection,  if the following conditions are satisfied:

1) $Z=\{z\in\C^{m} \mid h(z)=0 \mbox{ and } \mathrm{rank} \, d h (z)=m-n\}$, for some polynomial mapping  $h=(h_1,\ldots,h_{m-n})\colon\C^{m}\to\C^{m-n}$. 

2) $\dim \{ {(h_1)}_*=0,..., {(h_{m-n})}_*=0\}=m-n$, where ${h_i}_*$ denotes the highest homogeneous component of $h_i.$  

\noindent If $\deg h_i=a_i$, then we say that $Z$ has multi-degree $(a_1,...,a_{m-n}).$
\end{definition}

Now, let $X, Y \subset \C^{2n-1}$ be  $n$-dimensional strong complete intersections, we say: 
\[X=\{z\in\C^{2n-1} \mid f(z)=0 \mbox{ and } \mathrm{rank} \, d f (z)=n-1\},\]
\[Y=\{z\in\C^{2n-1} \mid g(z)=0 \mbox{ and } \mathrm{rank} \, d g (z)=n-1\},\] 
with $f=(f_1,\ldots,f_{n-1})\colon\C^{2n-1}\to\C^{n-1}$ and $g=(g_1,\ldots,g_{n-1})\colon\C^{2n-1}\to\C^{n-1}$ polynomial mapping. 

\begin{definition}\label{d:general} We say that $X$ and $Y$ are in general positions  if the following holds
	 \[\dim \left( \left(\bigcap_{1}^{n-1}\{(f_i)_*=0\}\right) \bigcap \left(\bigcap_{1}^{n-1}\{(g_i)_*=0\}\right)\right)=1.\]
\end{definition}

We consider
 \begin{equation}\label{eq:Phi1}
 \Phi: X\times Y \to \C^{2n-1} \mbox{ given by } \Phi(x,y)= \frac{1}{2}( x+y),
 \end{equation}
and we denote by $K_0(\Phi)$ the critical  values set of $\Phi$. If $p \in \C^{2n-1} \setminus K_0(\Phi)$, we set 
\[\mu_{X,Y}(p):=\chi (\Phi^{-1}(p)),\]
where $\chi(\Phi^{-1}(p))$ denotes the Euler Characteristic of   $\Phi^{-1}(p)$. The bifurcation set of $\Phi$ is denoted by $B(X,Y),$ see Definition \ref{d:bif}.  We also define the asymptotic part $B_\infty(X,Y)=B(X,Y)\setminus \overline{K_0(\Phi)},$ which contains only bifurcation points of $\Phi$ which comes from infinity. Note that it is not a standard definition.

As in \cite{JJR}, the set $B(X,Y)$ is also called a {\it symmetric defect} of $X$ and $Y$. The following result is obvious:
	
	\begin{proposition}\label{complete}
		Let $X,Y\subset \C^{2n-1}$ be $n$-dimensional strong complete intersections. Then, for a generic linear mapping $H: \C^{2n-1}\to\C^{2n-1}$, the sets $X$ and $H(Y)$ are in a general position.
\end{proposition}

In this note, we prove the following results related to the {\it symmetric defect}: 

\begin{theorem}\label{t:bif-euler} Let $X^n,Y^n\subset \C^{2n-1}$  be strong complete intersections in general positions. There exists  a number $\mu_{X,Y}$  such that, for any $p\notin \overline{K_0(\Phi)}$, we have  $p \notin   B_{\infty}(X,Y)$ if and only if  $\mu_{X,Y}=\mu_{X,Y}(p)$.  Moreover, the set $B(X,Y)$ is an algebraic set. 
\end{theorem}

\begin{theorem}\label{t:bif-euler1} Let $X^n,Y^n\subset \C^{2n-1}$  be strong complete intersections in general position  of multi-degrees 
	$(a_1,...,a_{n-1})$ and $(b_1,...,b_{n-1})$, respectively. Then   $B_\infty(X,Y)$  is contained in a hypersurface  $L_{\infty}(\Phi)$ such that \[ \deg L_{\infty}(\Phi) \leq D +d - \mu_{X,Y},\]
	where $D:=\deg \overline{\Sing(\Phi,L)\setminus\Sing(\Phi)}$, $d:=\mu(\Phi,L)$, $\mu(\Phi,L)$ is the  geometric degree of $(\Phi,L)$ and $L:\C^{2n-1}\to \C$ denotes a generic linear mapping.  In particular,  \[ \deg L_{\infty}(\Phi) \leq D -1 \leq \prod^{n-1}_{i=1} a_i \prod^{n-1}_{i=1} b_i (\sum^{n-1}_{i=1} a_i+b_i-2)-1.\]
	
\end{theorem}

Also, we have the following generic case:   

\begin{theorem}\label{ost1} Let $X^n,Y^n\subset \C^{2n-1}$  be strong complete intersections. For a generic linear mapping $H:\C^{2n-1} \to \C^{2n-1}$, the sets $X^n$ and $H(Y^n)$ are in general positions and the number $\mu_{X,H(Y)}$ does not depend of $H$. Moreover, for generic $H$ all sets $B(X,H(Y))$ are ambient homeomorphic, i.e, pairs $(\C^{2n-1}, B(X,H(Y)))$ and $(\C^{2n-1}, B(X,H'(Y)))$ are homeomorphic for generic 
	$H,H'.$
\end{theorem}	

This  allows us to define the set $\tilde{B}(X,Y):=B(X,H(Y))$ (where $H$ is  generic)  for every two strong complete intersections $X^n,Y^n\subset \C^{2n-1}$ (including the case $X=Y$) at least up to  homeomorphism.  It may happen that $\tilde{B}(X,Y)\not\cong B(X,Y)$.

Theorem \ref{t:bif-euler} and \ref{t:bif-euler1} are proved in \S\ref{ss:proof-1-2} and  Theorem \ref{ost1} in \S\ref{ss:last}. 

We point out that a version of Theorem \ref{t:bif-euler} was presented in  \cite{VT2} for a class of polynomial functions from $\C^{n}$ to  $\C^{n-1}$ and in \cite{JT} for  morphisms from $X$ to $Y$, with $X, Y$ affine varieties and $\dim X=\dim Y+1$. However, the algebraicity of the bifurcation set was not discussed in those papers.  
 
\section{Proofs of the results}\label{s:2}

\subsection{Non-properness set}

We recall some results on the non-properness set of a polynomial mapping. We follow \cite{JJR, jel, jel1}. Let $Z_1, Z_2$ be affine varieties over $\C$.

\begin{definition}
Let $f : Z_1 \rightarrow Z_2$ be a  generically finite (i.e. a
generic fiber is finite) and dominant polynomial mapping (i.e. $\overline{f(Z_1)}=Z_2$). We say that $f$ {\it is proper at} $y \in Z_2,$ if  there exists an open
neighborhood $U$ of $y$ such that  the mapping $ f\mid_{f^{-1}(U)} :f^{-1} (U)\rightarrow U$ is a proper mapping.
\end{definition}

We denote by  $S_f$ the set of points at which
the mapping $f$ is not proper. The set  
$S_f$ is called the {\it non-properness} set of $f.$ 

Let $Z\subset \C^m$ be an affine variety of codimension $k$. The degree of $Z$ is the number of common points of $Z$ and sufficiently general linear subspace $V \subset \C^m$ of dimension $k$. The degree of $Z$ is denoted by $\deg Z$. We have  $\deg \C^m =1$, for any $m \in \N$. 

The following result is useful (see for instance \cite{JJR, jel, jel1}):

\begin{theorem} Let $f=(f_1,...,f_n):Z\to \C^n $ be a generically finite dominant polynomial mapping, where $Z\subset \C^m$ is an affine $n$-dimensional variety with $\deg Z= D$. Then   $S_f$   is either empty or a hypersurface with  
		$${\rm deg}  \  S_f  \leq  \frac{
		D(\prod^n_{i=1} {\rm deg} \ f_i) - \mu (f)} {\min_{1\le i \le
			n}\deg\ f_i},$$ where $\mu(f)=(\C(X):\C(f_1,...,f_n))$ is the
	geometric degree of $f.$ 
\end{theorem}

Theorems \ref{t:bif-euler}, \ref{t:bif-euler1} and \ref{ost1} are related with the following definition: 

\begin{definition}\label{d:bif}  Let $f : Z_1\to Z_2$ be a polynomial mapping. The  bifurcation set of  $f$, denoted by $B(f)$,  is the smallest subset of $Z_2$ such that the restriction $f_|:Z_1 \setminus f^{-1}(B(f)) \to Z_2\setminus B(f)$  is a $C^{\infty}$ locally trivial   fibration. 
\end{definition}

Let $Z_1,Z_2$ be  smooth $n$-dimensional varieties over $\C$
and  $f : Z_1\to Z_2 $ a dominant polynomial mapping. We denote by 
 $\mu(f)$  the geometric degree of $f$.

The following useful result follows from the proof of  Theorem 2.4 of \cite{JJR}: 
 
\begin{theorem}[{\cite[Theorem 2.4]{JJR}}]\label{jk}
Let $Z$ be a smooth affine complex variety of dimension $n.$ Let
$f: Z \to \C^n$ be a dominant polynomial  mapping.
Then,  
\[B(f)=K_0(f)\cup S_f,\]
and
\[B(f)=\{y\in  \C^n : \# f^{-1}(y)\not=\mu(f)\}.\]

Moreover, $B(f)$ is either empty (and so $f$ is an isomorphism)  or
 a closed hypersurface. 
\end{theorem}

The next results are used in the proofs of Theorems \ref{t:bif-euler} and \ref{t:bif-euler1}: 

\begin{theorem}\label{fiber}
Let $X, Y$ be   complex irreducible algebraic varieties and
$f:X\to Y$ a  polynomial mapping. Let $W_1,..., W_r\subset X$
be closed subvarieties of $X.$ Then, there exists a nonempty
Zariski open  subset $U\subset Y$ such that, for every $y_1, y_2\in
U$, the $r+1$-tuples  $(f^{-1}(y_1), W_i\cap f^{-1}(y_1); i=1,...,r)$ and $(f^{-1}(y_2), W_i\cap f^{-1}(y_2), i=1,...,r) $ are homeomorphic.
\end{theorem}

\begin{proof}
We may assume that $Y$ is smooth and irreducible.
Let $X_1$ be an algebraic completion of $X.$  Take
$X_2:=\overline{graph(f)}\subset X_1\times {Y}.$ 

We can assume that $X\subset X_2.$ We have an induced mapping
$\overline{f}: X_2\to  {Y}$ such that
$\overline{f}_{|X}=f.$ Let $Z=X_2\setminus X.$ Denote by
$\overline{W_i},$ for $ i=1,...,r$, the closures of  $W_i$ in $X_2.$
Let $\cal R=\{Z, \overline{W_i}\cap Z,  \overline{W_i},  i=1,...,r) \}$ be a collection of algebraic
subvarieties of $X_2.$ There exists  a Whitney stratification ${\cal
S}$ of $X_2$ which is compatible with $\cal R.$

For any smooth strata  $S_i\in \cal S$, let $B_i$ be the set of
critical values of the mapping {$\overline{f}_{|S_i}$} and put 
$B=\overline{ \bigcup B_i}.$ Take $X_3=X_2\setminus
\overline{f}^{-1}(B).$  The restriction of the stratification
$\cal S$ to $X_3$ provides  a Whitney {stratification} which is
compatible with the family ${\cal R}':={\cal R}\cap X_3.$  We have
a proper mapping $f':=\overline{f}_{|X_3} : X_3\to
{Y}\setminus B$ which is a submersion on each stratum. By
  Thom first isotopy theorem, there is a trivialization of $f'$
which preserves the strata. It is an easy observation that this
trivialization gives a trivialization of the mapping $f:
X\setminus f^{-1}(B)\to Y\setminus B:=U.$ In particular the fibers
$f^{-1}(y_1)$ and $f^{-1}(y_2)$ are homeomorphic via a stratum
preserving homeomorphism. This means that the tuples  $(f^{-1}(y_1), W_i\cap f^{-1}(y_1); i=1,...,r)$ and $(f^{-1}(y_2), W_i\cap f^{-1}(y_2), i=1,...,r) $ are homeomorphic.
\end{proof}

\begin{corollary}\label{c:fiber}
Let $X, Y$ be   complex irreducible algebraic varieties and
$f:X\to Y$ a polynomial  mapping. Let $Z\subset X$
be a constructible subset  of $X.$ Then, there exists a nonempty
Zariski open subset $U\subset Y$ such that, for every $y_1, y_2\in
U$, the pairs  $(f^{-1}(y_1), Z\cap f^{-1}(y_1))$ and $(f^{-1}(y_2), Z\cap f^{-1}(y_2)) $ are homeomorphic.
\end{corollary}
 
\begin{proof}
Since $Z$ is constructible we have $Z=\bigcup^s_{i=1} P_i\setminus Q_i$, where $P_i, Q_i, \ i=1,..., s$ are Zariski closed. 
Now we can apply Theorem \ref{fiber}, where the family $\{W_i\}$ coincide with $\{P_i,Q_i; \ i=1,..., s\}$.
\end{proof}

\begin{theorem}\label{construct}
Let $X,Y$ be affine complex varieties and let
$f: X \to Y$ be a polynomial  mapping.
Then,  for every $k\in \Bbb N$, the set $A_k:=\{ y\in Y : \chi(f^{-1}(y))=k\}$ is a constructible set.
\end{theorem}

\begin{proof}
We proceed by induction on $d=\dim Y.$ If $d=0$, the result is obvious. Assume $d>0.$ Let $Y_1,..., Y_s$ be all $d-$dimensional components of $Y$ and let $Z_i=Y_i\setminus( \overline{Y\setminus Y_i}).$ By Theorem \ref{fiber}  there exists a nonempty open subset $U_i\subset Z_i$ such that
$\chi(f^{-1}(y))$ is constant for $y\in U_i.$ 
 Take $Y':=Y\setminus \bigcup^s_{i=1} U_i$ and $X':=f^{-1}(Y').$ By the induction hypothesis, the set $A'_k:=\{ y\in Y' : \chi(f^{-1}(y))=k\}$ is constructible. Let $U_{i_1},..., U_{i_r}$ be all subset among $U_1,..., U_s$ such that the fiber of $f$ over $U_i$ has Euler characteristic $k.$
Then $A_k=A'_k\cup  \bigcup^r_{k=1} U_{i_k}$ is a constructible subset of $Y.$
\end{proof}

\subsection{Proof of Theorems \ref{t:bif-euler} and \ref{t:bif-euler1}}\label{ss:proof-1-2}
 
The next technical lemma is motivated by \cite[Theorem 2.6]{VT2}. See Definition \ref{d:general} for the notion of general position. 

\begin{lemma}\label{l:L-proper}  Let $X^n,Y^n\subset \C^{2n-1}$  be strong complete intersections in general position. Then, there exists a linear mapping $L\colon \C^{2n-1}\to \C$ such that \[(\Phi,L): X\times Y  \to (\C^{2n-1}\times \C), (\Phi,L)(x,y)=(\Phi(x,y), L(x)) \] is a proper mapping. 
	%common points at infinity of $X$ and $Y$ is one-dimensional, i.e., the 
\end{lemma}

\begin{proof} By assumption,  $\mathcal{Z}_{X\cap Y}^{\infty}:=\left(\bigcap_{1}^{n-1}\{(f_i)_*=0\}\right) \bigcap \left(\bigcap_{1}^{n-1}\{(g_i)_*=0\}\right)\subset \C^{2n-1}$ is a  union of finitely many complex lines. Let $L\colon \C^{2n-1}\to \C$ be a linear mapping  such that $L(z)\neq 0,\forall z\in (\mathcal{Z}_{X\cap Y}^{\infty} \setminus \{0\}) $. We claim that \[(\Phi,L): X\times Y  \to (\C^{2n-1}\times \C),\]
$(\Phi,L)(x,y)=(\Phi(x,y), L(x)) $ is a proper mapping. 
	
	In fact, let $x(t)=(x_1(t),\ldots,x_{2n-1}(t))$ and  $y(t)=(y_1(t),\ldots,y_{2n-1}(t))$ be analytic curves such that $(x(t),y(t))\in X\times Y$,  $\Phi(x(t),y(t))\to p\in\C^{2n-1}$ and  $\|(x(t),y(t))\|\to\infty,$ as $t\to\infty$. It follows  that  $\|x(t)\| \to \infty $ and $\|y(t)\|\to\infty$, as $t\to\infty$.
	
	We set  \[x_i(t)=a_{\alpha_i}t^{\alpha_i}+\mbox{ lower terms, } y_i(t)=b_{\beta_i}t^{\beta_i}+\mbox{ lower terms},i=1,\ldots, 2n-1,\]
	$\alpha:=\max_{i}\{\alpha_i\}$, $\beta:=\max_{i}\{ \beta_i\}$, 
	$a':=(a_1',\ldots,a_{2n-1}')$ and $b':=(b_1',\ldots,b_{2n-1}'),$ where $a_i'=a_{\alpha_i}$ (resp., $b_i'=b_{\beta_i}$) if $\alpha=\alpha_i$ (resp., $\beta=\beta_i$) and $a'_i=0$ (resp., $b'_i=0$) if $\alpha>\alpha_i$ (resp., $\beta>\beta_i$).  
	
	Since $\Phi(x(t),y(t))\to p\in\C^{2n-1}$, we have $ \alpha= \beta$ and $a'=-b'$.  From the expansions:
	\[ f_j(x(t))=(f_j)_{\nu_j}(a')t^{\alpha \nu_j} + \mbox{ lower terms,}\]
	\[ g_j(y(t))=(g_j)_{\mu_j}(b')t^{\beta \mu_j } + \mbox{ lower terms}, j=1,\ldots, n-1, \]
	and since $f_j(x(t))\equiv  0, g_j(y(t))\equiv 0$ and $ \mathcal{Z}_{X\cap Y}^{\infty}$ is a homogeneous set, we have $a',b'\in \mathcal{Z}_{X\cap Y}^{\infty}$.  
	
	On the other hand, we have the following expansion for $L$: 
	\[ L(x(t))=L(a')t^{\alpha} + \mbox{ lower terms}.\]
	
	Since $a'\in \mathcal{Z}_{X\cap Y}^{\infty},$ we have $L(a')\neq 0$ and therefore  $\|L(x(t))\|\to \infty$. This shows that $(\Phi,L)$ is a proper mapping, as desired. 
\end{proof}

Directly from the  proof of Lemma \ref{l:L-proper}, we have that if $L\colon \C^{2n-1}\to \C$ is a linear mapping such that $L(z)\neq 0$ for any $z\in (\mathcal{Z}_{X\cap Y}^{\infty} \setminus \{0\})$, then the mapping $(\Phi,L): X\times Y  \to (\C^{2n-1}\times \C), (\Phi,L)(x,y)=(\Phi(x,y), L(x))$ is a proper mapping. 

Now, we prove Theorem \ref{t:bif-euler}. 

\begin{proof}[Proof of Theorem \ref{t:bif-euler}]
The proof is motivated by arguments from \cite[Theorem 3.1]{VT1} and \cite[Theorem 2.1]{VT2}.  

 From Lemma \ref{l:L-proper}, there is a linear mapping $L\colon \C^{2n-1}\to \C$ such that \[(\Phi,L): X\times Y  \to (\C^{2n-1}\times \C), (\Phi,L)(x,y)=(\Phi(x,y), L(x)) \] is a proper mapping. 

Let $p\notin  \overline{K_0(\Phi )}$ and consider the smooth algebraic set $\Phi^{-1}(p)$. The restriction $L_{|\Phi^{-1}(p)}$ is a proper mapping, since  $(\Phi,L)$ is a proper mapping and $\Phi^{-1}(p)$ is a closed set. It follows  that  $L_{|\Phi^{-1}(p)} (\Phi^{-1}(p))$ is a closed set and, since it is also a constructible set,  we have  that   $L_{|\Phi^{-1}(p)}(\Phi^{-1}(p))$ is an algebraic set (constructible closed sets are algebraic sets). Thus $L_{|\Phi^{-1}(p)}(\Phi^{-1}(p))=\C$. 

We have  $S_{(\Phi,L)}=\emptyset$, where $S_{(\Phi,L)}$ denotes the non-properness set of $(\Phi,L)$. It follows from Theorem \ref{jk} that $B(\Phi,L)=K_0( \Phi,L) $ and, consequently,  $(\Phi,L)$ is a ramified mapping of degree $d:=\mu(\Phi,L)$, where $\mu(\Phi,L)$ denotes the geometric degree of $(\Phi,L)$. Thus, for any  $p\notin \overline{K_0(\Phi )}$ and any regular value $c\in\C$ of the mapping $L_{|\Phi^{-1}(p)},$ we have  $  \# (L_{|\Phi^{-1}(p)})^{-1} (c)= \# \{\Phi^{-1}(p)\cap L^{-1}(z)\}=\mu(\Phi,L)=d$.  

Take $p_0\notin \overline{K_0(\Phi )}$ and consider a small enough neighborhood $W$ of $p_0$ such that $W \cap \overline{K_0(\Phi)}= \emptyset$. For any $p\in W$, we set
\[\Sigma_p:=\{z_1(p),\ldots,z_{r(p)}(p)\} \subset X \times Y\] 
and 
\[\Delta_p:=\{c_1(p),\ldots,c_{s(p)}(p)\} \subset \C,\] 
the  critical points set and critical values set of $L_{|\Phi^{-1}(p)}$ respectively.

As in the proofs of Theorem 3.1 of \cite{VT1}, Theorem 2.1  of \cite{VT2}, or from \cite[(III,32)]{kl1}, we obtain the following like Riemann-Hurwitz formula: 
\begin{equation}\label{eq:hur}
 \chi(\Phi^{-1}(p))= d + r(p) -\sum_{i=1}^{r(p)}\rho_i(p) = d - \sum_{i=1}^{r(p)} (\rho_i(p) -1) , 
\end{equation} 
where $ \rho_i(p)$ is the multiplicity of the point $x_i(p)$. These equalities imply that there is a small enough neighborhood of $p_0$ such that the following inequality holds  
\[ \chi(\Phi^{-1}(p_0)) \geq \chi(\Phi^{-1}(p)).   \]

Therefore,   there is a sufficiently small neighborhood $V:=B(p_0,\delta)$ such that $V\subset W$ and $\chi(\Phi^{-1}(p_0)) = \chi(\Phi^{-1}(p))$ for any $p \in V$ if and only if there is no point $x_i(p)$ that converges to infinity as $p$ converges to $p_0$.  

By similar arguments of \cite[Theorem 3.1]{VT1}, \cite[Theorem 2.1]{VT2} and \cite[Theorem 3.1]{jel04}, we  will show that $\chi(\Phi^{-1}(p_0)) = \chi(\Phi^{-1}(p)) \forall p \in V$ implies  $p_0\notin  B_{\infty}(X,Y)$.

As we have  said, $\chi(\Phi^{-1}(p_0)) = \chi(\Phi^{-1}(p)) \forall p \in V$ implies that   $ \cup_{p\in V}\Sigma_p $ is bounded and, since $(\Phi,L)$ is a proper mapping, that there exists    $r\in \R$ such that $K:= \Phi^{-1}(V) \cap L^{-1}(\overline{B(0,r)} )$ is a bounded set containing $\cup_{p\in V} \Sigma_p$. Then, we set: 
\[U_1 = \Phi^{-1}(V) \setminus \overline{K},\]
 and we will show that $\Phi$ restricted to $U_1$ is a fibration on $V$. 
 
Consider the (formal) system of differential equations: 
\[
\left(\begin{array}{cc}
 [df(x(t))]_{(n-1) \times (2n-1)}   &   [0]_{(n-1) \times (2n-1)}	 \\ 
\\
 {[0]}_{(n-1) \times (2n-1)} &   [dg(y(t))]_{(n-1) \times (2n-1)}	\\
 \\
 & \hspace{-2cm}{[d\Phi(x(t),y(t))]_{(2n-2) \times (4n-2)}} 	\\
  \\
 &  \hspace{-2.7cm} [d L(x(t),y(t))]_{1 \times (4n-2)}
\end{array}\right) \left(\begin{array}{c}
x'_1(t) \\ 
\vdots\\
x'_{2n-1}(t)\\
\\
y'_1(t) \\ 
\vdots\\
y'_{2n-1}(t)\\

\end{array}\right)  
 = \left(\begin{array}{l}
 [0]_{(n-1)\times 1}\\
 \\
{[0]}_{(n-1)\times 1}\\
\\
{[p-p_0]}_{(2n-1)\times 1}\\
 \\
 {[0]}_{1\times 1}
\end{array}\right).\]
We can solve the system, since the matrix on left size is invertible (at least in $U_1$) and we can write this system shortly as follows: 
\[z'(t) = V_1(p-p_0,z(t)),\]
with $z(t)=(x(t),y(t))$. From above constructions, it follows that: 
\[df(x(t))x'(t)=0, \hspace{.2cm} dg(y(t))y'(t)=0 ,\]
\[ d\Phi(V(p-p_0, z(t)))=p-p_0 , \hspace{.2cm}  dL(V(p-p_0,z(t)))=0. \] 

Let us consider the following differential equation: 
\begin{equation}\label{eq:triv}
z'(t) = V_1(p-p_0, z(t)), \,\, z(0)=\gamma \in \Phi^{-1}(p_0) \setminus \overline{B(0,R)}.
\end{equation}
It follows that if $z_{p-p_0,\gamma}(t)=(x(t),y(t))$ is a solution of \eqref{eq:triv} then $z_{p-p_0,\gamma}(t)\in X\times Y$,    $L(z_{p-p_0,\gamma}(t))$ is constant and  $\Phi(z_{p-p_0,\gamma}(t))=t(p-p_0) + p_0$. Then, from  definition of $U_1$ and since $(\Phi,L)$ is a proper mapping,  the trajectory of $z_{p-p_0,\gamma}$ is bounded and does not cross the boundary of $U_1$ for any $t\in [0,1]$. It follows that the trajectory $z_{p-p_0,\gamma}(t)$ is defined on whole $[0,1]$ and it is contained in $X\times Y$. 
 
 The phase flow $z_{p-p_0,\gamma}(t)$, $t\in[0,1]$,  transforms $\Phi^{-1}(p_0)\cap U_1$ into $\Phi^{-1}(t(p-p_0)+p_0)\cap U_1$ (in fact, by symmetry, it transforms $\Phi^{-1}(p_0)\cap U_1$ onto $\Phi^{-1}(t(p-p_0)+p_0)\cap U_1$). 
 The mapping \[ h: \Phi^{-1}(p_0)\cap U_1 \times V \ni (\gamma,p) \to z_{p-p_0,\gamma}(1) \in \Phi^{-1}(V)\cap U_1,\]
is a diffeomorphism. Thus $p_0 \notin B_{\infty}(X,Y). $   

Let $k$ be the Euler characteristic of a typical fiber of $\Phi.$ Hence a point $p\in \C^{2n-1}\setminus \overline{K_0(\Phi)}$ is typical if and only if $\chi(\Phi^{-1}(p))=k.$ By Theorem \ref{construct}, the set $A_k$ of such points is constructible in $\C^{2n-1}\setminus \overline{K_0(\Phi)}.$ Since $A_k$ is open,  it is Zariski open. Hence, $B(X,Y)=(\C^{2n-1}\setminus A_k)  \cup  \overline{K_0(\Phi)}$ is a Zariski closed set, i.e., $B(X,Y)$ is an algebraic set.
\end{proof}

\begin{definition}\label{d:L-ity} We set 
\[L_{\infty}(\Phi):=\{y\in Y \mid \exists \{x_k\} \subset \Sing(\Phi,L)\setminus \Sing(\Phi) \ \mbox{ s.t. } \|x_k\|\to \infty \mbox{ and } \Phi(x_k)\to y\}.\]
\end{definition}

Let $S:=\overline{\Sing(\Phi,L)\setminus \Sing(\Phi)}$.
We have that  $L_{\infty}(\Phi)$ is the non-properness set of the restriction  $\Phi_{|S}: S \to \C^{2n-1}$ and as a direct consequence of the proof of Theorem \ref{t:bif-euler} (after equation \ref{eq:hur}), it follows that:
\begin{equation}\label{eq:Bxy}
                  B_\infty (X,Y)\subset L_{\infty}(\Phi).
\end{equation}

 Theorem \ref{t:bif-euler1} follows from the next result:  

\begin{theorem}\label{l:deg} Let  $L$ be a generic linear mapping. Then $L_{\infty}(\Phi)$ is either empty or  a hypersurface such that:  
\begin{equation}
\deg L_{\infty}(\Phi) \leq D +d - \mu_{X,Y},  
\end{equation}  
where $D:=\deg \overline{\Sing(\Phi,L)\setminus \Sing(\Phi)}$, $d:=\mu(\Phi,L)$, where $\mu(\Phi,L)$ denotes the geometric degree of $(\Phi,L)$. 
\end{theorem}
\begin{proof} Consider $S:=\overline{\Sing(\Phi,L)\setminus \Sing(\Phi)}$. If $S  = \emptyset$, then $L_{\infty}(\Phi)= \emptyset$. Suppose $ S  \neq \emptyset$. 
We have   $\dim S = 2n-1$.  
 Consider  the restriction of $\Phi$ to $S$, which is an equidimensional mapping: 
\[ \Phi_{|S}: S\to\C^{2n-1}.\] Note that for $p\in \C^{2n-1}\setminus K_0(\Phi)$ the fiber $\Phi^{-1}(p)$ is smooth (hence reduced). Hence, if $Q$ is an
irreducible (one-dimensional) component of this fiber, then either the mapping $L_{|Q}: Q\to \C$ is an isomorphism or it has a finite number of critical points. In the first case, $Q$ is not contained in $S$. In the second case, $Q$ has only a finite number of points in common with $S$ (all critical points of $L_{|Q}$ are in S).
This means that the mapping $\Phi_|S$ is generically finite and so  dominant.

 Let $D:=\deg S$. We can apply Theorem 2.2. Note that the algebraic degree of $ \Phi$ is  $1$. Thus,
 
\[\deg L_{\infty}(\Phi) \leq D - \mu(\Phi_{|S}).\] 
In particular, \[\deg L_{\infty}(\Phi) \leq D - 1\] and, by \cite[Proposition 2.8]{JJR}, we have  $\deg L_{\infty}(\Phi)\leq \prod^{n-1}_{i=1} a_i \prod^{n-1}_{i=1} b_i (\sum^{n-1}_{i=1} a_i+b_i-2)-1.$

By \eqref{eq:hur}, we have   
\[-\mu(\Phi_{|S })= - r(p) \leq -\chi(\Phi^{-1}(p_0)) +d\]
and, consequently, we can obtain the estimation: 
\[\deg L_{\infty}(\Phi) \leq D - \mu(\Phi_{|S}) \leq 
D -\chi(\Phi^{-1}(p_0)) +d,\]
as desired.

\end{proof}

\subsection{Proof of Theorem \ref{ost1}}\label{ss:last}

Denote by $\mathcal{L}(\C^{2n-1}, \C^{2n-1})$ the set of linear mappings from $\C^{2n-1}$ to $\C^{2n-1}$.
Let us consider the mapping $\Psi: X\times Y\times \mathcal{L}(\C^{2n-1},\C^{2n-1})\ni (x,y,H)\mapsto (\frac{x+H(y)}{2}, H)\in \C^{2n-1}\times  \mathcal{L}(\C^{2n-1},\C^{2n-1}).$ The mapping $\Psi$ is a polynomial mapping, hence it has homeomorphic fibers outside the bifurcation set $B(\Psi).$ The Zariski closure $B:=\overline{B(\Psi)}$ is a proper algebraic subset of $\C^{2n-1}\times  \mathcal{L}(\C^{2n-1},\C^{2n-1}).$

 Put
$R=\{ H \in  \mathcal{L}(\C^{2n-1},\C^{2n-1}): \C^{2n-1} \times \{H\}\subset B\}.$ It is easy to see that the set $R$ is algebraic. Since the set $B$ is a proper algebraic subset of $\C^{2n-1}\times  \mathcal{L}(\C^{2n-1},\C^{2n-1})$,  the set $R$ is a proper algebraic subset of $\mathcal{L}(\C^{2n-1},\C^{2n-1}).$ Outside $R$ a generic fiber of the mapping $X\times Y \ni (x,y)\mapsto (\frac{x+H(y)}{2})\in \C^{2n-1}$ is homeomorphic to the generic fiber of $\Psi.$ We can finish the first part of theorem from Proposition \ref{complete}.

Now, we show that    $$(\C^{2n-1}, B(X,H(Y)))\ {\rm and} \ (\C^{2n-1}, B(X,H'(Y)))$$ are homeomorphic for generic $H,H'\in \mathcal{L}(\C^{2n-1},\C^{2n-1}).$ 

Denote $\Psi_H(x,y):=\Psi(x,y,H).$Given a mapping $f$, let $C(f)$ denote  the set of critical points of $f.$ It is easy to compute the $C(\Psi)=\bigcup_{H\in \mathcal{L}(\C^{2n-1},\C^{2n-1})} C(\Psi_H)\times \{H\}.$

In particular, if $\pi: \C^{2n-1}\times \mathcal{L}(\C^{2n-1},\C^{2n-1})\to \mathcal{L}(\C^{2n-1},\C^{2n-1})$ is a projection, then 
$\pi^{-1}(H)\cap K_0(\Psi)=K_0(\Psi_H).$ Since $K_0(\Psi)$ is a constructible set, it follows from Theorem  \ref{fiber} (or rather from their proof) that $\pi^{-1}(H)\cap \overline{K_0(\Psi)}=\overline{K_0(\Psi_H)}.$

Let $k$ be the Euler characteristic of a typical fiber of $\Psi$ and let $B_k=\{ y\in \C^{2n-1}\times \mathcal{L}(\C^{2n-1},\C^{2n-1}): \chi(\Psi^{-1}(y))\not=k\}.$ By Theorem \ref{construct}, the set $B_k$ is constructible. Moreover,   $B(\Psi_H)=\pi^{-1}(H)\cap (B_k\cup \overline{K_0(\Psi)})=B(X,H(Y)).$ Now, the result follows from  Corollary \ref{c:fiber}.

\section*{Acknowledgments} 

The first author was partially supported by the Fapemig-Brazil Grant APQ-02085-21 and CNPq-Brazil grant 301631/2022-0. The second author was partially supported by the grant of Narodowe Centrum Nauki number 2019/33/B/ST1/00755. The authors are grateful to professor M.A.S Ruas for many helpful discussions.

\end{document}